\documentclass[reqno]{amsart}
\usepackage{amsmath}
  \usepackage{paralist}
  \usepackage{graphics} 
  \usepackage{epsfig} 

\newtheorem{theorem}{Theorem}[section]

\theoremstyle{definition}

\def\no{|\!|}

\newcommand{\cB}{\mathcal B}
\newcommand{\cE}{\mathcal E}

\newcommand{\cH}{\mathcal H}

\newcommand{\R}{\mathbb R}
\newcommand{\C}{\mathbb C}

\newcommand{\N}{\mathbb N}

\newcommand{\st}{\,:\,}

\newcommand{\EE}{\mathbb E}

\newcommand{\ra}{\rightarrow}

\title[Normal stability]
     {On normal stability for nonlinear parabolic equations}

\author[Jan Pr\"uss, Gieri Simonett and Rico Zacher]{}


 \email{jan.pruess@mathematik.uni-halle.de}
 \email{gieri.simonett@vanderbilt.edu}
 \email{rico.zacher@mathematik.uni-halle.de}

\thanks{The second author is partially supported by NSF grant DMS-0600870.
The third author was partially supported by the Deutsche Forschungsgemeinschaft (DFG)}

\begin{document}
\maketitle

\centerline{\scshape Jan Pr\"uss }
\medskip
{\footnotesize
 \centerline{Institut f\"ur Mathematik,
             Martin-Luther-Universit\"at Halle-Witten\-berg}
   \centerline{D-06120 Halle, Germany}
} 

\medskip

\centerline{\scshape Gieri Simonett }
\medskip
{\footnotesize
 \centerline{Department of Mathematics, Vanderbilt University}
   \centerline{Nashville, TN 37240, USA}
} 

\medskip

\centerline{\scshape Rico Zacher}
\medskip
{\footnotesize
 \centerline{Institut f\"ur Mathematik, 
             Martin-Luther-Universit\"at Halle-Witten\-berg}
   \centerline{D-06120 Halle, Germany}
} %

\bigskip


\begin{abstract}
We show convergence of solutions to equilibria for 
quasilinear and fully nonlinear
parabolic evolution equations in situations where the set of
equilibria is non-discrete, but forms a finite-dimensional
$C^1$-manifold which is normally stable. 

\end{abstract}
\section{Introduction}
In this short note we consider quasilinear as well as fully 
nonlinear parabolic equations
and we study convergence of solutions towards equilibria
in situations where the set of equilibria forms a
$C^1$-manifold.
\smallskip\\
Our main result can be summarized as follows:
suppose that for a nonlinear evolution equation we
have a $C^1$-{\em manifold of equilibria} $\cE$ such that at a
point $u_*\in\cE$, the kernel $N(A)$ of the linearization $A$ is
isomorphic to the tangent space of $\cE$ at $u_*$, the eigenvalue
$0$ of $A$ is semi-simple, and the remaining spectral part of the
linearization $A$ is stable. Then solutions starting nearby $u_*$
exist globally and converge to some point on $\cE$.
This situation occurs frequently in
applications. We call it the {\em generalized principle of
linearized stability}, and the equilibrium $u_*$ is then termed
{\em normally stable}.

A typical example for this situation to occur is the
case where the equations under consideration involve symmetries,
i.e.\ are invariant under the action of a Lie-group. 

The situation where the set of equlibria forms a $C^1$-manifold
occurs for instance in phase transitions \cite{ES98,PrSi06},
geometric evolution equations \cite{EMS98,ES99},
free boundary problems in fluid dynamics \cite{FR02, GP97},
stability of traveling waves \cite{PSZ08},
and models of tumor growth, to mention just a few.

A standard method to handle situations as described above is to
refer to {\em center manifold theory}. In fact, in that situation
the center manifold of the problem in question will be unique, and
it coincides with $\cE$ near $u_*$. Thus the so-called {\em
shadowing lemma} in center manifold theory implies the result.
Center manifolds are well-studied objects in the theory of nonlinear
evolution equations. For the parabolic case we refer to the
monographs \cite{Hen81,Lun95}, and to 
the publications 
\cite{BaJo89, BJL00, DPLu88,LPS08,Mie91,Sim95,VI92}.

However, the theory of center manifolds is a technically difficult matter.  Therefore it seems desirable to have a simpler, direct approach to the generalized principle of linearized stability which avoids the technicalities of center manifold theory.

Such an approach has been introduced in \cite{PSZ08}
in the framework of $L_p$-maximal regularity. 
It turns
out that within this approach the effort to prove convergence
towards equilibria in the normally stable case 
is only slightly larger than that for the proof
of the standard linearized stability result - which is simple.

The purpose of this paper is to extend the approach given in
\cite{PSZ08} to cover a broader setting and a broader class of 
nonlinear parabolic equations, including fully nonlinear equations. 
This approach is flexible and general enough to reproduce the
results contained in \cite{Cui07,EMS98,ES98,ES99,FR02,GP97,PrSi06,PSZ08},
and it will have applications to many other problems.

Our approach makes use of the concept of maximal regularity in an essential way. As general references for this theory we refer to the monographs
\cite{Ama95,DHP03,Lun95}.
\section{Abstract  nonlinear problems in a general setting}
\noindent Let $X_0$ and $X_1$ be Banach spaces, 
and suppose that $X_1$ is densely embedded in $X_0$.
Suppose that $F:U_1\subset X_1\to X_0$ satisfies
\begin{equation}
\label{F}
F\in C^k(U_1,X_0),\quad k\in \N,\ k\ge 1,
\end{equation}
where $U_1$ is an open subset of $X_1$.
Then we consider the autonomous (fully) nonlinear problem
\begin{equation} 
\label{FN1}
\dot{u}(t)+F(u(t))=0,\quad t>0, \quad u(0)=u_0,
\end{equation}
for $u_0\in U_1$.
In the sequel we use the notation $|\cdot|_j$ to denote the norm
in the respective spaces $X_j$ for $j=0,1$.
Moreover, for any normed space $X$,
$B_X(u,r)$ denotes the open ball in $X$ with radius $r>0$ around $u\in X$.
\smallskip\\
\noindent
Let $ \cE\subset U_1$ denote the set of  equilibrium solutions 
of \eqref{FN1}, which means that
$$
u_\ast\in\cE \quad \mbox{ if and only if }\quad F(u_\ast)=0.
$$
Given an  element $u_*\in\cE$, we assume that $u_*$ is
contained in an $m$-dimensional manifold of equilibria. This means that there
is an open subset $U\subset\R^m$, $0\in U$, and a $C^1$-function
$\Psi:U\rightarrow X_1$ such that
\begin{equation}
\label{manifold}
\begin{aligned}
& \bullet\
\text{$\Psi(U)\subset \cE$ and $\Psi(0)=u_*$,} \\
& \bullet\
 \text{the rank of $\Psi^\prime(0)$ equals $m$, and} \\
& \bullet\
\text{$F(\Psi(\zeta))=0,\quad \zeta\in U.$}
\end{aligned}
\end{equation}
We assume further that near $u_*$ there are no other equilibria
than those given by $\Psi(U)$,
i.e.\ $\cE\cap B_{X_1}(u_*,{r_1})=\Psi(U)$, for some $r_1>0$.
\smallskip\\
\noindent
Let $u_\ast\in\cE$ be given and set $A:=F^\prime (u_\ast)$.
Then we assume that $A\in\cH(X_1,X_0)$, 
by which we mean that $-A$, considered as a linear
operator in $X_0$ with domain $X_1$, generates a strongly continuous
analytic semigroup $\{e^{-At};\,t\ge 0\}$ on $X_0$. In particular we
may take the graph norm of $A$ as the norm in $X_1$.

For the deviation $v:=u-u_*$ from $u_*$, equation \eqref{FN1}
can be restated as 
\begin{equation} 
\label{FN2}
\dot{v}(t)+Av(t)=G(v(t)),\quad t>0, \quad v(0)=v_0,
\end{equation}
where $v_0=u_0-u_*$, and $G(z):=Az-F(z+u_*)$,
$z\in V_1:=U_1-u_\ast$.
It follows from \eqref{F}
that
$G\in C^k(V_1,X_0)$.
Moreover, we have
$G(0)=0$ and $G^\prime(0)=0$.
Setting
$\psi(\zeta)=\Psi(\zeta)-u_*$ results in the following equilibrium
equation for problem \eqref{FN2}
\begin{equation}
\label{equilibrium-psi}
A\psi(\zeta)=G(\psi(\zeta)),\quad \mbox{ for all }\;\zeta\in U.
\end{equation}
Taking the derivative with respect to $\zeta$ and using
the fact that $G^\prime(0)=0$ we conclude that
$A\psi^\prime(0)=0$ and this implies that
the tangent space of $\cE$ at
$u_\ast$ is contained in $N(A)$, the kernel of $A$.

For $J=[0,a)$, $a\in (0,\infty]$, we consider a pair of Banach
spaces $(\EE_0(J),\EE_1(J))$ such that $\EE_0(J)\hookrightarrow
L_{1,{\rm loc}}(J;X_0)$ and
\[
\EE_1(J)\hookrightarrow H^1_{1,{\rm loc}}(J;X_0)\cap L_{1,{\rm loc}}(J;X_1),
\]
respectively.
Denoting by $X_\gamma=\gamma\EE_1$ the trace space of $\EE_1(J)$ 
we assume that
\begin{itemize}
\item[(A1)]
$\gamma\EE_1$ is independent of $J$, and the embedding
$\EE_1(J)\hookrightarrow BU\!C(J;X_\gamma)$ holds. 
In addition, we assume that
 there is a constant $c_0> 0$ independent of $J=[0,a)$, 
$a\in (0,\infty]$, such that
\begin{equation}
\label{trace-0}
\sup_{t\in J} \no w(t)\no_{\gamma}\le c_0\no w\no_{\EE_1(J)},\quad
\mbox{for all}\; w\in \EE_1(J),\; w(0)=0.
\end{equation}
\end{itemize}
We refer to \cite[Section III.1.4]{Ama95} for further information
on trace spaces.
Moreover, we assume that
\begin{itemize}
\item[(A2)]
 $\tilde{w}\in \EE_1(J)$ and
$|w(t)|_0\le|\tilde{w}(t)|_1$, $t\in J$, 
imply $\no w\no_{\EE_0(J)}\le \no\tilde{w}\no_{\EE_1(J)}$; 
for $\omega>0$ fixed, there exists a constant $c_1>0$
not depending on $J$ and such that
\begin{equation} 
\label{FNR}
\begin{split}
&\int_J e^{-\omega s}|w(s)|_1\,ds
\le c_1\no w\no_{\EE_1(J)},\quad
\mbox{for all}\;w\in \EE_1(J),\\
&\int_t^\infty e^{-\omega s}|w(s)|_1\,ds
\le c_1 e^{-\omega t}\no w\no_{\EE_1(\R_+)},\quad
\mbox{for all}\;w\in \EE_1(\R_+)\text{ and } t\ge 0.
\end{split}
\end{equation}
\end{itemize}
Our {\em key assumption} is that
$(\EE_0(J),\EE_1(J))$ is a pair of maximal regularity for $A$.
To be more precise we assume that
\begin{itemize}
\item[(A3)]
the linear Cauchy problem
$\dot w +Aw=g,\ w(0)=w_0 $
has for each $(g,w_0)\in \EE_0(I)\times \gamma\EE_1(I)$
a unique solution $w\in\EE_1(I)$, where 
$I=[0,T]$ is a finite interval.
\end{itemize}
We impose the following assumption for the sake of convenience.
For all examples that we have in mind the condition
can be derived from (A3).

Suppose that $\sigma(A)$, the spectrum of $A$, admits
a decomposition
$\sigma(A)=\sigma_s\cup \sigma^\prime $, where
$\sigma_s\subset\{z\in\C:{\rm Re}\, z>\omega\}$
for some $\omega>0$ and 
$\sigma^\prime\subset\{z\in\C:{\rm Re}\, z\le 0\}$.
Let $P_s$ denote the spectral projection 
corresponding to the spectral set $\sigma_s$.
Then we assume that
\begin{itemize}
\item[(A4)] there exists a constant $M_0>0$ such that for any $J=[0,a)$, $a\in (0,\infty]$,
any $\sigma\in [0,\omega]$, and any function $g$ with $e^{\sigma
t}P_s g\in \EE_0(J)$ there is a unique solution $w$
 of $\dot{w}+A_s w=P_s g$, $t\in J$, $w(0)=0$,
satisfying
\[
\no e^{\sigma t}w\no_{\EE_1(J)}\le M_0\no e^{\sigma t}P_s g\no_{\EE_0(J)};
\]
there exists a constant $M_1>0$ such that for
any $J=[0,a)$, $a\in (0,\infty]$, and for any $z\in X_\gamma$
there holds
\[
\no e^{\sigma t}e^{-A_s t}P_s z\no_{\EE_1(J)}
+\sup_{t\in J}|e^{\sigma t}e^{-A_s t}P_s z|_\gamma
\le M_1 |P_s z|_{\gamma},\quad \sigma\in [0,\omega].
\]
\end{itemize}
We again refer to \cite[Chapter III]{Ama95}
for more background information on the notion of 
maximal regularity.
In order to cover the case $X_\gamma\neq X_1$ we assume
the following {\em structure condition} on the nonlinearity $G$:
\begin{itemize}
\item[(A5)] there exists a uniform constant $C_1$ such that for any 
$\eta>0$ there is $r>0$ such that
\begin{equation*}
\hspace{1cm}
|G(z_1)-G(z_2)|_0\le C_1(\eta+|z_2|_1)|z_1-z_2|_1,\quad
z_1,\,z_2\in X_1\cap B_{X_\gamma}(0,r).
\end{equation*}
\end{itemize}
Observe that condition (A5) trivially holds in the case
$X_\gamma= X_1$, since $G'(0)=0$. A short computation shows that
condition (A5) is also satisfied if $F$ has a quasilinear 
structure, i.e. if
\begin{equation}
\label{quasilinear}
F(u)=B(u)u+f(u)\quad\text{for $u\in U_\gamma$},\quad
(B,f)\in C^1(U_\gamma,\cB(X_1,X_0)\times X_0),
\end{equation}
where $U_\gamma\subset X_\gamma$ is an open set.
\smallskip\\
Lastly, concerning {\em solvability} of the nonlinear problem (\ref{FN2}) we will assume that
\begin{itemize}
\item[(A6)] given $b>0$ there exists $r_2>0$ such that for any
$v_0\in B_{X_\gamma}(0,r_2)$ problem (\ref{FN2}) admits a
unique solution $v\in \EE_1([0,b])$.
\end{itemize}
Note that since $v=0$ is
an equilibrium of (\ref{FN2}), condition (A6) is satisfied
whenever one has existence and uniqueness of local solutions in the
described class as well as continuous dependence of the maximal time
of existence on the initial data.
\medskip\\
We conclude this section by describing three important examples
of admissible pairs $(\EE_0(J),\EE_1(J))$.
\medskip\\
\noindent
{\bf Example 1:} ($L_p$-maximal regularity.)\\
In our first example, the spaces $(\EE_0(J),\EE_1(J))$
are given by
\begin{equation}
\EE_0(J):=L_p(J;X_0),
\quad
\EE_1(J):=H^1_p(J;X_0)\cap L_p(J;X_1).
\end{equation}
The trace space is a
real interpolation space given by
$\gamma\EE_1=X_\gamma=(X_0,X_1)_{1-1/p,p}$ and we have
$\EE_1(J)\hookrightarrow BU\!C(J;X_\gamma)$,
see for instance \cite[Theorem III.4.10.2]{Ama95}.
For a proof of \eqref{trace-0} we refer to 
\cite[Proposition 6.2]{PSS07}. This yields
Assumption (A1).
For Assumption (A2) we note that
\begin{equation*}
\begin{split}
\int_J e^{-\omega s}|w(s)|_1\,ds
\le c_1\big(\int_J |w(s)|_1^p\,ds\big)^{1/p}
\le c_1 \no w\no_{\EE_1(J)}
\end{split}
\end{equation*}
for all $w\in \EE_1(J)$ by H\"older's inequality. Moreover,
\begin{equation*}
\begin{split}
\int_t^\infty e^{-\omega s}|w(s)|_1\,ds
\le \big(\int_t^\infty e^{-\omega s p^\prime}\,ds\big)^{1/p^\prime}
\big(\int_t^\infty |w(s)|^p\,ds\big)^{1/p^\prime}
\le c_1 e^{-\omega t}\no  w\no_{\EE_1(\R_+)}
\end{split}
\end{equation*}
for $t\ge 0$ and $w\in \EE_1(\R_+)$.
We refer to \cite{DHP03,KW04, PrSi07}, \cite[Section III.4.10]{Ama95}
and the references therein for conditions
guaranteeing that the crucial Assumption (A3)
on maximal regularity is satisfied.
It is clear that the property of maximal regularity
is passed on from $A$ to $A_s$ in the spaces 
$\EE^s_0(J):=L_p(J;X_0^s)$, $\EE^s_1(J):=H^1_p(J;X_0^s)\cap L_p(J;X_1^s)$,
and this implies Assumption (A4), see for instance
\cite[Remark III.4.10.9(a)]{Ama95}.
Assumption (A5) is satisfied in case that
the nonlinear mapping $F$ has a {\em quasilinear structure}, see \cite{PSZ08}.
Assumption (A6) follows in case that $F$ has a quasilinear structure
from (A3) and \cite[Theorem 3.1]{Pru03}, 
see also \cite[Theorem 2.1, Corollary 3.3]{Am05}.
We remark that the case of $L_p$-maximal regularity has been
considered in detail in \cite{PSZ08}.
\medskip
\goodbreak
\noindent
{\bf Example 2:} ({Continuous maximal regularity}).\\
Let $J=[0,a)$ with $0<a\leq\infty$
and set $\dot J:=(0,a)$.
For $\mu\in (0,1)$ and $X$ a Banach space we set
\begin{equation*}
\begin{split}
&BU\!C_{1-\mu}(J;X):=\big\{u\in C(\dot{J};X):[t\mapsto t^{1-\mu}u]\in
BU\!C(\dot{J};X),\\
& \hspace{6cm}\lim_{t\to 0^+} t^{1-\mu}|u(t)|_X=0\big\},\\
&BU\!C_0(J;X):=BU\!C(J;X). \\
\end{split}
\end{equation*}
$BU\!C_{1-\mu}(J;X)$ is turned into a Banach space by the norm
\begin{equation*}
\no u\no_{C_{1-\mu}(J;X)}:=\sup_{t\in \dot{J}} t^{1-\mu}|u(t)|_X,
\quad \mu\in (0,1].
\end{equation*}
Finally, we set 
$ BU\!C_{1-\mu}^1(J;X):=\{u\in C^1(\dot{J};X):u,\,\dot{u}\in
BUC_{1-\mu}(J;X)\}.$
With these preparations we define
\begin{equation}
\label{continuous}
\begin{split}
\EE_0(J):&=BU\!C_{1-\mu}(J;X_0), \\
\EE_1(J):&=BU\!C_{1-\mu}^1(J;X_0)\cap BU\!C_{1-\mu}(J;X_1)
\end{split}
\end{equation}
endowed with the canonical norms.

Supposing that
$\cH(X_1,X_0)\neq\emptyset$ the trace space $\gamma\EE_1$ is the
continuous interpolation space
$\gamma\EE_1=(X_0,X_1)_{\mu,\infty}^0=:D_A(\mu)$, and we have the embedding
$\EE_1(J)\hookrightarrow BU\!C(J;\gamma\EE_1)$, see
\cite[Theorem III.2.3.3]{Ama95}. 
A proof for estimate \eqref{trace-0} can
be found in \cite[Lemma 2.2(c)]{ClSi01},
and this shows that Assumption (A1) is satisfied. 
Assumption~(A2) holds as
\begin{equation*}
\begin{split}
\int_J e^{-\omega s}|w(s)|_1\,ds=
\int_J \frac{e^{-\omega s}}{s^{1-\mu}}s^{1-\mu}|w(s)|_1\,ds
\le c_1 \no w\no_{C_{1-\mu}(J;X_1)}
\le c_1 \no w\no_{\EE_1(J)}
\end{split}
\end{equation*}
for all $w\in \EE_1(J)$, and 
\begin{equation*}
\begin{split}
\int_t^\infty e^{-\omega s}|w(s)|_1\,ds=
\int_t^\infty \frac{e^{-\omega s}}{s^{1-\mu}}s^{1-\mu}|w(s)|_1\,ds
\le c_1 e^{-\omega t}\no  w\no_{\EE_1(\R_+)}
\end{split}
\end{equation*}
for $t\ge 0$ and $w\in \EE_1(\R_+)$.
\medskip\\
It turns out that maximal regularity
cannot hold in the class \eqref{continuous}
if $X_1\neq X_0$ and $X_0$ is reflexive.
On the other side, there is an interesting class
of spaces $(X_0, X_1)$ where Assumption (A3)
is indeed satisfied for the pair
$(\EE_0(J),\EE_1(J))$ given in \eqref{continuous},
see \cite{Ang90, ClSi01, DPrGr79, Lun95}
and \cite[Theorem III.3.4.1]{Ama95}.
$A_s$ inherits the property of maximal regularity from $A$,
and this implies Assumption (A4), see
\cite[Remark III.3.4.2(b)]{Ama95}.
Assumption~(A5) holds
in the case $\mu=1$ for any function $G\in C^1(U_1,X_0)$
with $G(0)=G^\prime(0)=0$.
It  also holds for $\mu\in (0,1)$ if the nonlinear function
$F$ given in \eqref{FN1} satisfies \eqref{quasilinear}.
\smallskip\\
If $\mu=1$ and $k\ge 1$ then it follows from (A3) and 
\cite[Theorem 2.7, Corollary 2.9]{Ang90}, 
see also \cite[Section 8.4]{Lun95},
that Assumption (A6) is satisfied.
\smallskip\\
If $\mu\in (0,1)$, $k\ge 1$ and $F$ has a {\em quasilinear} structure,
see \eqref{quasilinear},
then Assumption (A6) follows from (A3) and \cite[Theorem 5.1]{ClSi01},
see also \cite[Theorem 6.1]{ClSi01}. 
\bigskip\\
\goodbreak
\noindent
{\bf Example 3:} (H\"older maximal regularity.)\\
Suppose $\rho\in (0,1)$,  $I\subset\R_+$, $J\subset\R_+ $
are intervals with $0\in J$.
Then we set 
\begin{equation*}
\begin{split}
[u]_{C^\rho(I;X)}&:=
\sup\Big\{\frac{|u(t)-u(s)|}{|t-s|^\rho}: s,t\in I,\ s\neq t\Big\},\\
[\![u]\!]_{C^\rho_\rho(J;X)}
&:=\sup_{2\varepsilon\in\dot J}    \varepsilon^\rho[u]_{C^\rho([\varepsilon,2\varepsilon];X)},                                                  
\end{split}
\end{equation*}
and
\begin{equation*}
\begin{split}
\no u\no_{C^\rho_\rho(J;X)}&:=\no u\no_{BC(I;X)}
+[\![u]\!]_{C^\rho_\rho(J;X)}, \\
BC^\rho_{\rho}(J;X)&:=\{u\in C^\rho(J;X): 
\no u\no_{C^\rho_\rho(J;X)}<\infty\}. 
\end{split}
\end{equation*}                                                               
Moreover, we set
\begin{equation*}
\begin{split}
BU\!C^\rho_\rho(J;X):=
\{u\in BU\!C(J;X)\cap BC^\rho_\rho(J;X): 
\lim_{\varepsilon\to 0^+} \varepsilon^\rho 
[u]_{C^\rho_\rho([\varepsilon,2\varepsilon];X)}=0\}
\end{split}
\end{equation*}
and equip it with the norm 
$\no \cdot \no_{C^\rho_\rho(J;X)}$.
For the pair $(\EE_0(J),\EE_1(J))$ we take
\begin{equation}
\begin{split}
\label{H1}
&\EE_0(J):=BU\!C^\rho_\rho(J;X_0), \\
&\EE_1(J):=BU\!C^{1+\rho}_\rho(J;X_0)\cap BU\!C^\rho_\rho(J;X_1),
\end{split}
\end{equation}
where 
$BU\!C^{1+\rho}_\rho(J;X)
:=\{u\in BU\!C^\rho_\rho(J;X_0): \dot u\in BU\!C^\rho_\rho(J;X_0)\}$.
The spaces in \eqref{H1} are given their canonical norms,
turning them into Banach spaces.
\smallskip\\
We have $\gamma\EE_1(J)=X_1$ and it is clear from the definition
of (the norm of) $\EE_1(J)$ that 
$\EE_1(J)\hookrightarrow BU\!C(J,X_1)$,
and that \eqref{trace-0} is satisfied
for any $w\in\EE_1(J)$.
This shows that Assumption (A1) holds.
By similar arguments as above we see that
Assumption (A2) is satisfied as well.
\smallskip\\
For the crucial Assumption (A3)
we refer to \cite[Theorem III.2.5.6]{Ama95} with
$\mu=1$; see also \cite[Corollary 4.3.6(ii)]{Lun95}.
It is worthwhile to mention that this  
maximal regularity result is true for {\em any} $A\in\cH(X_1,X_0)$ and any pair $(X_0,X_1)$. 
Assumption (A4) follows then as above, see \cite[Theorem III.2.5.5]{Ama95}.
Assumption (A5) holds for {\em any} function 
$G\in (U_1,X_0)$ with $G(0)=G^\prime(0)=0$.
\smallskip\\
Finally, it follows from Theorem 8.1.1 and Theorem 8.2.3 in
\cite{Lun95} that Assumption (A6)
holds for the fully nonlinear problem \eqref{FN1}
in case that $k\ge 2$.
(In fact, it suffices to require that
the derivative $F^\prime$ of $F$ be locally Lipschitz continuous.)
\section{The main result}        
In this section we state and prove our main theorem about
convergence of solutions for the nonlinear equation
\eqref{FN1} towards equilibria.
\begin{theorem}
\label{th:1}
Let $u_*\in X_1$ be an
equilibrium of (\ref{FN1}), and assume that the above 
conditions (A1)-(A6) are satisfied. 
Suppose that $u_*$ is normally stable, i.e.\ assume that
\begin{itemize}
\item[(i)] near $u_*$ the set of equilibria $\cE$ is a $C^1$-manifold in $X_1$ of dimension $m\in\N$,
\item[(ii)] \, the tangent space for $\cE$ at $u_*$ is given by $N(A)$,
\item[(iii)] \, $0$ is a semi-simple eigenvalue of
$A$, i.e.\ $ N(A)\oplus R(A)=X_0$,
\item[(iv)] \, $\sigma(A)\setminus\{0\}\subset \C_+=\{z\in\C:\, {\rm Re}\, z>\omega\}$ for some $\omega>0$.
\end{itemize}
Then $u_*$ is stable in $X_\gamma$, and there exists $\delta>0$ such
that the unique solution $u(t)$ of \eqref{FN1} with initial
value $u_0\in X_\gamma$ satisfying $|u_0-u_*|_\gamma<\delta$
exists on $\R_+$ and converges at an exponential rate 
to some $u_\infty\in\cE$ in $X_\gamma$ as $t\rightarrow\infty$.
\end{theorem}
\begin{proof}
The proof to Theorem 2.1 will be carried out in several steps, as follows.
\medskip\\
(a) We denote by $P_l$, $l\in\{c,s\}$, the spectral projections corresponding
to the spectral sets $\sigma_s$ and $\sigma_c:=\{0\}$, respectively, and let
$A_l=P_l A P_l$ be the part of $A$ in 
$X_0^l=P_l(X_0)$ for $l\in\{c,s\}$.
Note that $A_c=0$.
We set $X_j^l:=P_l(X_j)$ for $l\in\{c,s\}$ and $j\in\{0,\gamma,1\}$.
It follows from our assumptions that
$X^c_0=X^c_1$.
In the following we set $X^c:=X^c_0$ and equip
$X^c$ with the norm of $X_0$.
Moreover, we take as a norm on $X_j$
\begin{equation}
\label{norm-decomposition}
|v|_j:=|P_c v|_0 + |P_s v|_j\quad\text{for}\quad j=0,\gamma,1.
\end{equation}
(b) Next we show that
 the manifold $\cE$ can be represented
as the (translated) graph of a function 
$\phi:B_{X^c}(0,\rho_0)\to X_1^s$ in a neighborhood
of $u_\ast$.
In order to see this we consider the mapping
\begin{equation*}
g:U\subset \R^m\to X^c,\quad
g(\zeta):=P_c\psi(\zeta),\quad \zeta\in U.
\end{equation*}
It follows from our assumptions that
$g^\prime(0)=P_c\psi^\prime(0):\R^m\to X^c$
is an isomorphism.
By the inverse function theorem, $g$ is a
$C^1$-diffeomorphism of a neighborhood of $0$ in $\R^m$
onto a neighborhood, say $B_{X^c}(0,\rho_0)$, of $0$ in $X^c$.
Let $g^{-1}:B_{X^c}(0,\rho_0)\to U$ be the inverse mapping.
Then $g^{-1}:B_{X^c}(0,\rho_0)\to U$ is $C^1$ and $g^{-1}(0)=0$.
Next we set
$\Phi(x):=\psi(g^{-1}(x))$ for $x\in B_{X^c}(0,\rho_0)$
and we note that
\begin{equation*}
\Phi\in C^1(B_{X^c}(0,\rho_0),X_1^s), \quad
\Phi(0)=0,
\quad \{u_\ast +\Phi(x) \st x\in B_{X^c}(0,\rho_0)\}=\cE\cap W,
\end{equation*}
where $W$ is an appropriate neighborhood of $u_\ast$ in $X_1$.
Clearly,
\begin{equation*}
P_c \Phi(x)=((P_c\circ \psi)\circ g^{-1})(x)=
(g\circ g^{-1})(x)=x,\quad x\in  B_{X^c}(0,\rho_0),
\end{equation*}
and this yields
$\Phi(x)=P_c\Phi(x)+P_s\Phi(x)=x+P_s\Phi(x)$ for
$x\in B_{X^c}(0,\rho_0)$.
Setting $\phi(x):=P_s\Phi(x)$ we conclude that
\begin{equation}
\label{phi}
\phi\in C^1(B_{X^c}(0,\rho_0),X_1^s),\quad \phi(0)=\phi^\prime (0)=0,
\end{equation}
and that
$
\{u_\ast +x+\phi(x) \st x\in B_{X^c}(0,\rho_0)\}=\cE\cap W,
$
where $W$ is a neighborhood of $u_\ast$ in $X_1$.
This shows that the manifold $\cE$
can be represented
as the (translated) graph of the function $\phi$ in a neighborhood
of $u_\ast$. Moreover,
the tangent space of $\cE$ at $u_\ast$
coincides with $N(A)=X^c$.
By applying the projections $P_l$, $l\in\{c,s\}$, to equation \eqref{equilibrium-psi}
and using that $x+\phi(x)=\psi(g^{-1}(x))$
for $x\in B_{X^c}(0,\rho_0)$, and that $A_c\equiv 0$,
we obtain the following equivalent system of equations
for the equilibria of \eqref{FN2}
\begin{equation}
\label{equilibria-phi}
P_cG(x+\phi(x))=0,\quad
P_s G(x+\phi(x))=A_s\phi(x),
\quad x\in B_{X_c}(0,\rho_0).
\end{equation}
Finally, let us also agree that $\rho_0$ has already been chosen small enough
so that
\begin{equation}
\label{estimate-phi}
|\phi^\prime(x)|_{\cB(X^c,X_1^s)}\le 1 ,\quad
|\phi(x)|_1\le  |x|,\quad x\in B_{X^c}(0,\rho_0).
\end{equation}
This can always be achieved, thanks to \eqref{phi}.
\medskip\\
\noindent
(c) 
Introducing the new variables
\begin{equation*}
\begin{aligned}
&x=P_c v=P_c (u-u_*), \\
&y=P_sv-\phi(P_cv)=P_s(u-u_*)-\phi(P_c (u-u_*))
\end{aligned}
\end{equation*}
we then
obtain the following system of evolution equations
in $X^c\times X^s_0 $
\begin{equation}
\label{system}
\left\{
\begin{aligned}
\dot{x}=T(x,y),      \quad &x(0)=x_0, \\
\dot{y}+A_sy=R(x,y), \quad &y(0)=y_0,\\
\end{aligned}
\right.
\end{equation}
with $x_0=P_cv_0$ and $y_0=P_sv_0-\phi(P_cv_0)$,
where the functions $T$ and $R$ are given by
\begin{equation*}
\begin{aligned}
&T(x,y)=P_c G(x+\phi(x)+y), \\
&R(x,y)=P_sG(x+\phi(x)+y)-A_s\phi(x)-\phi^\prime(x)T(x,y).
\end{aligned}
\end{equation*}
Using the equilibrium equations \eqref{equilibria-phi}, the
expressions for $R$ and $T$ can be rewritten as
\begin{equation}
\label{R-T}
\begin{aligned}
&T(x,y)=P_c \big(G(x+\phi(x)+y)-G(x+\phi(x))\big), \\
&R(x,y)=P_s \big(G(x+\phi(x)+y)-G(x+\phi(x))\big)-\phi^\prime(x)T(x,y).
\end{aligned}
\end{equation}
Equation \eqref{R-T}
immediately yields
\begin{equation*}
\label{R=T=0}
T(x,0)=R(x,0)=0\quad\text{for all }\ x\in B_{X^c}(0,\rho_0),
\end{equation*}
showing that
the equilibrium set $\cE$ of \eqref{FN1}
near $u_*$ has been reduced to the set
$ B_{X^c}(0,\rho_0)\times \{0\}\subset X^c\times X^s_1$.
\medskip\\
Observe also that there is a unique correspondence between the solutions of \eqref{FN1}
close  to $u_*$ in $X_\gamma$ and those of (\ref{system}) close to $0$.
We call  system \eqref{system} the {\em normal form}
of \eqref{FN1} near
its normally stable equilibrium $u_*$.
\medskip\\
(d)
Taking $z_1=x+\phi(x)+y$ and $z_2=x+\phi(x)$
it follows from (A5), \eqref{estimate-phi} and \eqref{R-T} that
\begin{equation}
\label{estimate-R-T}
\begin{aligned}
|T(x,y)|,\ |R(x,y)|_0 \le C_1\big(\eta +|x+\phi(x)|_1\big)|y|_1
\le \beta |y|_1,
\end{aligned}
\end{equation}
with $\beta:=C_2(\eta+r)$,
where the constants $C_1$ and $C_2$ are independent of
$\eta,r$ and $x,y$, provided that $x\in \bar B_{X^c}(0,\rho)$,
$y\in \bar B_{X^s_\gamma}(0,\rho)\cap X_1$
and $\rho\in (0,r/3]$ with $r<3\rho_0$.
Suppose that $\eta$ and, accordingly, $r$ were already chosen
small enough so that
\begin{equation}
\label{beta}
M_0\beta= M_0C_2(\eta +r)\le 1/2.
\end{equation}
(e)
Suppose now that $v_0\in B_{X_\gamma}(0,\delta)$, where
$\delta<r_2$ will be chosen later. By (A6), problem
(\ref{FN2}) has a unique solution on some maximal interval of
existence $[0,t_*)$.
Let $\eta$ and $r$ be fixed so that \eqref{beta} holds
and set $\rho=r/3$. Let then $t_1$ be the exit time for the ball
$\bar B_{X_\gamma}(0,\rho)$, that is
\[
t_1:=\sup\{t\in(0,t_*):|v(\tau)|_\gamma\le \rho,\,\tau\in[0,t]\}.
\]
Suppose $t_1<t_*$ and set $J_1=[0,t_1)$.
The definition of $t_1$  implies that $|x(t)|\le \rho $ for all $t\in J_1$.
Assuming wlog that the embedding constant of $X_1\hookrightarrow X_\gamma$
is less or equal to one, we obtain from \eqref{norm-decomposition}
\begin{equation*}
\begin{split}
\rho\ge |v(t)|_\gamma
&=|x(t)+\phi(x(t))+y(t)|_\gamma 
    =|x(t)|+|\phi(x(t))+y(t)|_\gamma \\
&\ge |x(t)|+|y(t)|_\gamma-|\phi(x(t))|_\gamma 
\ge |y(t)|_\gamma
\end{split}
\end{equation*}
for $t\in J_1$, since $\phi(x)$ is non-expansive for $|x|\le \rho_0$.
In conclusion we have shown that
$|x(t)|,|y(t)|\le \rho$ for all $t\in J_1$,
so that the estimate \eqref{estimate-R-T}
holds for $(x(t),y(t))$, $t\in J_1$.
Then, by (A4) and \eqref{estimate-R-T}, 
we have for $\sigma\in[0,\omega]$
\begin{equation*}
\begin{split}
\no e^{\sigma t}y\no_{\EE_1(J_1)}
&\le \no e^{\sigma t}e^{-A_s t}y_0\no_{\EE_1(J_1)}
+M_0\no e^{\sigma t}R(x,y)\no_{\EE_0(J_1)} \\
&
\le M_1 |y_0|_\gamma+M_0\beta\no e^{\sigma t}y\no_{\EE_1(J_1)},
\end{split}
\end{equation*}
which implies
\begin{equation}
 \label{FN4}
 \no e^{\sigma t}y\no_{\EE_1(J_1)}\le 2M_1|y_0|_\gamma,\quad
 \sigma\in[0,\omega],
\end{equation}
thanks to \eqref{beta}.
Using (A1), (A4) and (\ref{FN4}) we
then have for $t\in J_1$,
\begin{equation*}
\begin{split}
|e^{\omega t}y(t)|_\gamma
&\le |e^{\omega t}y(t)-e^{\omega t}e^{-A_s t}y_0|_\gamma
+|e^{\omega t}e^{-A_s t}y_0|_\gamma\\
&\le c_0\no e^{\omega t}y-e^{\omega t}  
e^{-A_st}y_0\no_{\EE_1(J_1)}+M_1|y_0|_\gamma\\
& \le (3c_0M_1+M_1)|y_0|_\gamma,
\end{split}
\end{equation*}
which yields with $M_2=3c_0 M_1+M_1$,
\[
|y(t)|_\gamma\le M_2e^{-\omega t}|y_0|_\gamma,\quad
t\in J_1.
\]
Using (\ref{FNR}) we deduce further from the equation for $x$ and
the estimate for $T$ in \eqref{estimate-R-T},
and from \eqref{beta}--\eqref{FN4} that
\begin{equation*}
\begin{split}
|x(t)|\,& \le |x_0|+\int_0^t|T(x(s),y(s))|\,ds
 \le |x_0|+\beta \int_0^t |y(s)|_1\,ds \\ 
& \le |x_0|+\beta c_1 \no e^{\omega t}y\no_{\EE_1(J_1)}
\le |x_0|+M_3|y_0|_\gamma,\quad t\in J_1,
\end{split}
\end{equation*}
where $M_3=M_1 c_1/M_0$. Since $v(t)=x(t)+\phi(x(t))+y(t)$, 
the previous estimates and \eqref{estimate-phi} imply
that for some constant $M_4\ge 1$,
\begin{equation*}
|v(t)|_\gamma\le M_4|v_0|_\gamma,
\quad t\in J_1.
\end{equation*}
Choosing $\delta=\min\{\rho,r_2\}/(2M_4)$, we have
$|v(t_1)|_\gamma\le \min\{\rho,r_2\}/2$, a contradiction
to the definition of $t_1$, and hence
$t_1=t_*$. The above argument then yields uniform bounds
$\no v\no_{\EE_1(J)}\le C$ and $\sup_{t\in J}|v(t)|_\gamma\le
r_2/2$ for all $J=[0,a)$ with $a<t_*$. In view of (A6), it follows
that $t_*=\infty$. 
\medskip\\
(f)
Repeating the above estimates on the interval $[0,\infty)$ we obtain
\begin{equation}
\label{y-gamma-infty} |x(t)|\le |x_0|+M_3|y_0|_\gamma,\quad
|y(t)|_\gamma\le M_2 e^{-\omega t} |y_0|_\gamma, \quad t\in
[0,\infty),
\end{equation}
for $v_0\in B_{X_\gamma}(0,\delta)$.
Moreover, 
$
\lim_{t\ra\infty} x(t)= x_0 +\int_0^\infty T(x(s),y(s))ds=:x_\infty
$
exists since the integral is absolutely convergent.
This yields existence of
\begin{equation*}
v_\infty:=\lim_{t\ra\infty} v(t)=\lim_{t\ra\infty} x(t)+\phi(x(t))+y(t)=x_\infty+\phi(x_\infty).
\end{equation*}
Clearly, $v_\infty$ is an equilibrium
for equation \eqref{FN2}, and 
$u_\infty:=u_\ast+ v_\infty\in\cE$ is an equilibrium for \eqref{FN1}.
It follows from (A2), the estimate for $T$ in \eqref{estimate-R-T},
and from \eqref{FN4} that
\begin{equation*}
\begin{aligned}
|x(t)-x_\infty|=\Big|\int_t^\infty T(x(s),y(s))\,ds\Big|  
&\le \beta \int_t^\infty |y(s)|_1\,ds \\
&\le \beta c_1 e^{-\omega t}\no e^{\omega t}y\no_{\EE_1(\R_+)} 
\le M_4  e^{-\omega t} |y_0|_\gamma. 
\end{aligned}
\end{equation*}
This shows that $x(t)$ converges to $x_\infty$
at an exponential rate.
Due to \eqref {estimate-phi}, \eqref{y-gamma-infty}
and the exponential estimate for $|x(t)-x_\infty|$
we now get for the solution $u(t)=u_\ast+v(t)$ of \eqref{FN1}
\begin{equation}
\label{exponential-v}
\begin{aligned}
|u(t)-u_\infty|_\gamma
&=|x(t)+\phi(x(t))+y(t)-v_\infty|_\gamma \\
&\le |x(t)-x_\infty|_\gamma +|\phi(x(t))-\phi(x_\infty)|_\gamma +|y(t)|_\gamma \\
&\le (2M_4+M_2)e^{-\omega t}|y_0|_\gamma\\
&\le Me^{-\omega t}|P_sv_0-\phi(P_cv_0)|_\gamma\,,
\end{aligned}
\end{equation}
thereby completing the proof of the second part of
Theorem~\ref{th:1}. Concerning stability,
note that given $r>0$ small enough we may choose
$0<\delta\le r$ such that the solution starting in
$B_{X_\gamma}(u_*,\delta)$ exists on $\R_+$ and stays within $B_{X_\gamma}(u_*,r)$.
\end{proof}
\noindent
\textbf{Remarks:}
(a) Theorem~\ref{th:1} has been proved in \cite{PSZ08}
in the setting of $L_p$-maximal regularity,
and applications to quasilinear parabolic
problems with nonlinear boundary conditions, to the
Mullins-Sekerka problem, and to the stability of
travelling waves for a quasilinear parabolic equation
have been given.
\smallskip\\
(b)
It has been shown in \cite{PSZ08} by means of examples
that conditions (i)--(iii) in Theorem~\ref{th:1}
are also necessary in order to get convergence of solutions towards
equilibria $u_\infty\in\cE$.

\end{document}